\documentclass{amsart}

\font\fiverm=cmr5

\let\<=\langle
\let\>=\rangle
\let\\=\backslash

\let\sse=\subseteq
\let\noi=\noindent

\let\veps=\varepsilon
\let\limply=\Longrightarrow

\def\0{\{0\}}

\def\span{{\kern.5pt{\rm span}\kern1pt}}
\def\smallfrac#1#2{{\textstyle{\frac{#1}{#2}}}}
\def\conv{{\;\longrightarrow\;}}
\def\wconv{{{\buildrel_{\scriptstyle w}\over\conv}}}
\def\sconv{{{\buildrel_{\scriptstyle s}\over\conv}}}
\def\uconv{{{\buildrel_{\scriptstyle u}\over\conv}}}

\def\sslash{\hbox{{\fiverm/}}}
\def\notconv{{{\conv\kern-13pt\slash}\kern9pt}}
\def\notuconv{{{\uconv\kern-13pt\sslash}\kern9pt}}
\def\notsconv{{{\sconv\kern-13pt\sslash}\kern9pt}}
\def\notwconv{{{\wconv\kern-13pt\sslash}\kern9pt}}
\def\query{{{\buildrel_{^{\scriptstyle ?}}\over\limply}}}

\def\A{{\mathcal A}}
\def\B{{\mathcal B}}
\def\C{{\kern.5pt\mathcal C}}

\def\M{{\mathcal M}}

\def\Oe{{\mathcal O}}
\def\Pe{{\mathcal P}}
\def\R{{\mathcal R}}
\def\U{{\mathcal U}}
\def\X{{\mathcal X}}

\def\BX{{\B[\X]}}

\def\CC{{\mathbb C\kern.5pt}}
\def\FF{{\mathbb F\kern.5pt}}
\def\DD{{\mathbb D\kern.5pt}}
\def\NN{{\mathbb N\kern.5pt}}
\def\RR{{\mathbb R\kern.5pt}}
\def\TT{{\mathbb T\kern.5pt}}
\def\ZZ{{\mathbb Z\kern.5pt}}

\newsymbol\varnothing 203F
\newsymbol\subsetneqq 2324
\newsymbol\thicksim 2373
\let\void=\varnothing

\def\newmatrix#1{\null\,\vcenter{
		\baselineskip=8pt\mathsurround=-0pt\ialign{
		\hfil ${##}$
		\hfil &&
		\hfil ${##}$
		\hfil \crcr
		\mathstrut \crcr
		\noalign{\kern-\baselineskip}#1 \crcr
		\mathstrut \crcr
		\noalign{\kern-\baselineskip} \crcr }}\!}

\newtheorem{theorem}{Theorem}
\newtheorem{lemma}{Lemma}
\newtheorem{corollary}{Corollary}
\newtheorem{proposition}{Proposition}

\theoremstyle{definition}

\newtheorem{remark}{Remark}

\numberwithin{theorem}{section}
\numberwithin{lemma}{section}
\numberwithin{corollary}{section}
\numberwithin{proposition}{section}
\numberwithin{conjecture}{section}
\numberwithin{definition}{section}
\numberwithin{remark}{section}
\numberwithin{question}{section}

\begin{document}

\vglue-72pt\noindent
\hfill{\it Mathematical Proceedings of the Royal Irish Academy}\/,
{\bf 118A} (2018) 47--56

\vglue20pt
\title{On Weak Supercyclicity I}
\author{C.S. Kubrusly}
\address{Applied Mathematics Department, Federal University,
         Rio de Janeiro, RJ, Brazil}
\email{carloskubrusly@gmail.com}
\author{B.P. Duggal}
\address{8 Redwood Grove, Northfield Avenue, Ealing, London W5 4SZ,
         United Kingdom}
\email{bpduggal@yahoo.co.uk}
\renewcommand{\keywordsname}{Keywords}
\date{December 7, 2017; Revised April 11, 2018}

\begin{abstract}
This paper provides conditions (i) to distinguish weak supercyclicity
form supercyclicity for operators acting on normed and Banach spaces, and
also (ii) to ensure when weak supercyclicity implies weak stability.
\end{abstract}

\maketitle

\vskip-15pt\noi
\section{Introduction}

The purpose of this paper is to establish conditions to distinguish weak
supercyclicity form supercyclicity for operators acting on normed spaces, and
also to provide conditions on weakly supercyclic operators to ensure weak
stability$.$ Sections 2 and 3 deal with notation and terminology, and also
offer a broad view on supercyclicity in the weak and norm topologies$.$
Auxiliary results are considered in Section 4$.$ Thus Sections 2, 3 and 4
present a brief survey on supercyclicity emphasizing the role played by weak
supercyclicity$.$ The new results appear in Sections 5 and 6$.$ Theorem 5.1
characterizes weakly l-sequentially supercyclic vectors that are not
supercyclic for a power bounded operator, and Theorems 6.1 and 6.2
exhibit a criterion to extend the results on supercyclicity and strong
stability of \cite{AB} to weak l-sequential supercyclicity and weak
stability$.$ The main result is Theorem 6.2, and special classes of operators
on Hilbert space are considered Corollaries 5.2 and 6.1.

\vskip6pt\noi
\section{Notation and Terminology}

Let $\FF$ stand either for the complex field $\CC$ or for the real field
$\RR$, and let $\X$ be an infinite-dimensional normed space over $\FF$$.$
Let $A^-$ denote the closure (in the norm topology of $\X$) of a set
${A\sse\X}.$ A subspace of $\X$ is a {\it closed}\/ linear manifold of $\X.$
If $\M$ is a linear manifold of $\X$, then its closure $\M^-$ is a subspace$.$
By an operator on $\X$ we mean a linear bounded (i.e., continuous)
transformation of $\X$ into itself$.$ Let $\BX$ be the normed algebra of all
operators on $\X.$ A subspace $\M$ of $\X$ is invari\-ant for an operator
${T\in\BX}$ (or $T$-invariant) if ${T(\M)\sse\M}$, and it is nontrivial if
${\0\ne\M\ne\X}.$ Let $\|T\|$ stand for the induced uniform norm of $T$ in
$\BX$.

\vskip6pt
An operator ${T\in\BX}$ is power bounded if ${\sup_{n\ge0}\|T^n\|<\infty}$,
it is a contraction if ${\|T\|\le1}$ (i.e., if ${\|T^nx\|\le\|x\|}$ for
every ${x\in\X}$ and every integer ${n\ge0}$), and it is an isometry if
$\|T^nx\|=\|x\|$ for every ${x\in\X}$ and every integer ${n\ge0}.$ (On a inner
product space, a unitary operator is precisely an invertible isometry)$.$ An
operator ${T\in\BX}$ is weak\-ly or strongly stable (notation$:$
${T^n\kern-1pt\wconv O}$ or ${T^n\kern-1pt\sconv O}$) if the $\X$-valued
sequence $\{T^nx\}_{n\ge0}\kern-1pt$ converges weakly or strongly (i.e., or in
the norm topology of $\X$) to zero for every ${x\in\X}.$ In other words, if
${T^nx\wconv0}$, which means ${f(T^nx)\to0}$ for every $f$ in the dual
$\X^*$ of $\X$, for every $x$ in $\X$; or ${T^nx\conv0}$, which means
${\|T^nx\|\to0}$, for every $x$ in $\X$$.$ An operator ${T\in\BX}$ is
uniformly stable (notation$:$ ${T^n\kern-1pt\uconv O}$) if the $\BX$-valued
sequence $\{T^n\}_{n\ge0}$ converges to the null operator $O$ in the induced
uniform norm of $\BX$, which means ${\|T^n\|\to0}$.

\vskip6pt
The orbit $\Oe_T(y)$ of a vector ${y\in\X}$ under an operator ${T\in\BX}$ is
the set
$$
\Oe_T(y)={\bigcup}_{n\ge0}T^ny=\big\{T^ny\in\X\!:\,n\in\NN_0\big\},
$$
\goodbreak\noi
where $\NN_0$ denotes the set of nonnegative integers --- we write
${\bigcup}_{n\ge0}T^ny$ for the set
${\bigcup}_{n\ge0}T^n(\{y\})={\bigcup}_{n\ge0}\{T^ny\}.$ The orbit $\Oe_T(A)$
of a set ${A\sse\X}$ under an operator $T$ is the set
$\Oe_T(A)=\bigcup_{n\ge0}T^n(A)=\bigcup_{y\in A}\Oe_T(y).$ For any set
${A\sse\X}$ let $\span\kern1pt A$ be the (linear) span of $A$ (the linear
manifold spanned by $A$)$.$ The projective orbit of a vector ${y\in\X}$ under
an operator ${T\in\BX}$ is the orbit of the one-dimensional space spanned by
the singleton $\{y\}$; that is, it is the orbit of span of $\{y\}$:
$$
\Oe_T(\span\{y\})={\bigcup}_{n\ge0}T^n(\span\{y\})
=\big\{\alpha T^ny\in\X\!:\;\alpha\in\FF,\;n\in\NN_0\big\}.
$$
A vector $y$ in $\X$ is a {\it cyclic vector}\/ for an operator $T$ in $\BX$
if $\X$ is the smallest invariant subspace for $T$ containing $y.$
Equivalently, ${y\in\kern-1pt\X}$ is a cyclic vector for $T$ if its orbit
spans $\X$:
\vskip-2pt\noi
$$
\big(\span\Oe_T(y)\big)^-\!=\X.
$$
\vskip2pt\noi
Still equivalently, $y$ is a cyclic vector for $T$ if
$\{p(T)y\!:p\;\hbox{is a polynomial}\}^-\!=\X$, which means
$\{Sy\!:{S\in\Pe(T)}\}^-\!=\X$, where $\Pe(T)$ is the algebra of all
polynomials in $T$ with scalar coefficients$.$ An operator ${T\in\BX}$ is
a {\it cyclic operator}\/ if it has a cyclic vector$.$ Stronger forms of
cyclicity are defined as follows$.$ A vector $y$ in $\X$ is a
{\it supercyclic vector}\/ for an operator $T$ in $\BX$ if its projective
orbit is dense in $\X$ in the norm topology; that is, if
$$
\Oe_T(\span\{y\})^-\!=\X.
$$
An operator $T$ in $\BX$ is a {\it supercyclic operator}\/ if it has a
supercyclic vector$.$ Moreover, a vector $y$ in $\X$ is a {\it hypercyclic
vector}\/ for an operator $T$ in $\BX$ if the orbit of $y$ is dense in $\X$
in the norm topology; that is, if
$$
\Oe_T(y)^-\!=\X.
$$
An operator $T$ in $\BX$ is a {\it hypercyclic operator}\/ if it has
a hypercyclic vector.

\vskip6pt
Versions in the weak topology of the above notions read as follows$.$ Let
$A^{-w}$ denote the weak closure of a set ${A\sse\X}$ (i.e., the closure of
$A$ in the weak topology of $\X).$ A vector $y$ in $\X$ is a {\it weakly
cyclic vector}\/ for an operator $T$ in $\BX$ if
$$
\big(\span\Oe_T(y)\big)^{-w}=\X,
$$
and $T$ in $\BX$ is a {\it weakly cyclic operator}\/ if it has a weakly cyclic
vector$.$ (Weak cyclicity, however, collapses to plain cyclicity according
to Remark 3.1(f) below)$.$ A vector $y$ in $\X$ is a {\it weakly supercyclic
vector}\/ for an operator $T$ in $\BX$ if
$$
\Oe_T(\span\{y\})^{-w}=\X,
$$
and $T$ in $\BX$ is a {\it weakly supercyclic operator}\/ if it has a weakly
supercyclic vector$.$ A vector $y$ in $\X$ is a
{\it weakly hypercyclic vector}\/ for an operator $T$ in $\BX$ if
$$
\Oe_T(y)^{-w}=\X,
$$
and $T$ in $\BX$ {\it is a weakly hypercyclic operator}\/ if it has a weakly
hypercyclic vector$.$ (For a treatise on hypercyclicity see \cite{GP}.)

\vskip3pt\noi
\begin{remark}
Although we will not deal with $n$-supercyclicity in this paper, we just
pose definitions for sake of completeness: an operator ${T\in\BX}$ is
{\it $n$-supercyclic}\/ ({\it weakly $n$-supercyclic}\/) it there is an
$n$-dimensional subspace of $\X$ whose orbit under $T$ is dense (weakly
dense) in $\X.$ So a one-supercyclic (weakly one-supercyclic) is precisely
a supercyclic (weakly supercyclic) operator$.$ For each
${n\kern-.5pt\ge\kern-.5pt1}$ there are exam\-ples of $n$-supercyclic
operators that are not ${(n\kern-1pt-\!1)}$-supercyclic (see, e.g.,
\cite[p.2]{BM}).
\end{remark}

\section{Preliminaries}

{\it If\/ $T$ has a cyclic vector, then\/ $\X$ is separable}\/ (because $\X$
is spanned by the count\-able set $\Oe_T(y)$ --- see, e.g.,
\cite[Proposition 4.9]{EOT}), and so cyclic operators exist only on
separable normed space; in particular, supercyclic and hypercyclic operators
(as well as weakly cyclic, weakly supercyclic, and weakly hypercyclic)
operators only exist on separable normed spaces (thus separability is not
an assumption, but a consequence of cyclicity).

\vskip3pt\noi
\begin{remark}
For each vector ${y\in\X}$ consider its punctured projective orbit; that is,
its projective orbit under an operator ${T\in\BX}$ excluding the origin,
$$
\Oe_T(\span\{y\})\\\0
=\big\{\alpha T^ny\in\X\!:\;\alpha\in\FF\\\0,\;n\in\NN_0\big\}\\\0.
$$
For each ${z\in\Oe_T(\span\{y\})}\\\0$ the set
${\Oe_T(\span\{y\})\\\Oe_T(\span\{z\})}$ is a finite union of one-dimensional
subspaces of $\X.$ So if ${y\in\X}$ is supercyclic (weakly supercyclic)
for $T\kern-1pt$, then every $z\in{\Oe_T(\span\{y\})\\\0}$ is supercyclic
(weakly supercyclic) for $T$:

\vskip4pt
\begin{description}
\item{$\kern-6pt$(a)$\kern1pt$}
if a vector $y$ is supercyclic or weakly supercyclic for $T\kern-1pt$, then so
is $\alpha T^my$ for every ${0\ne\alpha\in\FF}$ and every integer ${m\ge0}.$
\end{description}
\vskip2pt\noi
In particular, item (b) below on supercyclic and weakly supercyclic
vectors are immediately verified, and item (c) is straightforward since
supercyclicity is defined in terms of denseness in the norm topology
(which is metrizable).

\vskip4pt
\begin{description}
\item{$\kern-6pt$(b)$\kern1pt$}
Every nonzero multiple of a supercyclic (weakly supercyclic) vector for an
operator is again a supercyclic (weakly supercyclic) vector for it$.$ Hence
an operator is supercyclic (weakly supercyclic) if and only if any nonzero
multiple of it is supercyclic (weakly supercyclic).

\vskip4pt
\item{$\kern-6pt$(c)$\kern3pt$}
Since the norm topology is metrizable, a nonzero vector $y$ in $\X$ is a
supercyclic vector for an operator $T$ in $\BX$ if and only if for every
${x\in\X}$ there exists an $\FF$-valued sequence $\{\alpha_k\}_{k\ge0}$
(which depends on $x$ and $y$ and consists of nonzero numbers) such that for
some subsequence $\{T^{n_k}\}_{k\ge0}$ of $\{T^n\}_{n\ge0}$ the $\X$-valued
sequence $\{\alpha_kT^{n_k}y\}_{k\ge0}$ converges to $x$ (in the norm
topology)$:$ 
$$
\alpha_kT^{n_k}y\conv x.
$$
(i.e., ${\|\alpha_kT^{n_k}y\kern-1pt-\kern-1ptx\|\kern-1pt\to\kern-1pt0}).$
If\/ ${\FF=\CC}$ and $\{\alpha_k\}_{k\ge0}$ is constrained to be $\RR$-valued,
then the notion of supercyclicity is referred to as
$\RR$-{\it supercyclicity}\/ \cite{BBP}.

\vskip4pt\noi
\item{$\kern-6pt$(d)$\kern1pt$}
If a vector $y$ is supercyclic (or hypercyclic) for an operator $T\kern-1pt$,
then $y$ is su\-percyclic (or hypercyclic) for every positive power $T^n$ of
$T$ \cite[Theorems 1 and 2]{Ans}$.$ Hence if an operator is supercyclic (or
hypercyclic) then so is every positive power $T^n$ of it.
\end{description}

\vskip4pt
\begin{description}
\item{$\kern-6pt$(e)$\kern2pt$}
If an operator $T$ is supercyclic (weakly supercyclic), then the set of all
supercyclic (weakly supercyclic) vectors for it is dense (weakly dense).

\vskip6pt\noi
Indeed, if $y$ is a supercyclic (weakly supercyclic) vector for $T\kern-1pt$,
then the punctured projective orbit $\Oe_T(\span\{y\})\\\0$ is dense (weakly
dense) in $\X$$.$ But according to item (a) $\Oe_T(\span\{y\})\\\0$ is included
in the set of all supercyclic (weakly supercyclic) vectors for $T\kern-1pt$,
and so the set of all supercyclic (weakly supercyclic) vectors is dense
(weakly dense) in $\X$ as well$.$ In fact, weakly dense can be extended to
dense (in the norm topology) \cite[Proposition 2.1]{San1}.

\vskip4pt
\item{$\kern-6pt$(f)$\kern2pt$}
Denseness (in the norm topology) implies weak denseness (and the converse
fails)$.$ However, if a set $A$ is convex, then ${A^-\!=A^{-w}}$ (e.g., see
\cite[Theorem 2.5.16]{Meg}), and so cyclicity coincides with weak cyclicity,
since span is convex.

\vskip4pt
\item{$\kern-6pt$(g)$\kern2pt$}
An operator has a nontrivial invariant subspace if and only if it has a
nonzero noncyclic vector$.$ Since cyclicity coincides with weak cyclicity,
it follows that if every nonzero vector in $\X$ is cyclic for $T\kern-1pt$
in any form of cyclicity discussed here (see Diagram 1 below), then $T$ has
no nontrivial invariant subspace.
\end{description}
\end{remark}

\vskip4pt
Definitions of Section 2 and Remark 3.1(f) ensure the following relations.
\vskip2pt\noi
$$
\newmatrix{
& {\rm Hypercyclic} & \limply & {\rm Weakly\;\;Hypercyclic} \phantom{\Big|} \cr
& \big\Downarrow    &         & \big\Downarrow \phantom{\Big|}              \cr
& {\rm Supercyclic} & \limply & {\rm Weakly\;\;Supercyclic} \phantom{\Big|} \cr
& \big\Downarrow    &         & \big\Downarrow \phantom{\Big|}              \cr
& {\rm Cyclic}      & \iff    & {\rm Weakly\;\/Cyclic}.\phantom{\Big|}     \cr}
$$
\vskip1pt\noi
$\kern-9pt$\centerline {Diagram 1.}
\vskip6pt\noi
Thus cyclicity (i.e., cyclicity the norm topology), which coincides with
weak cyclicity, is the weakest form (in the sense that it is implied by the
other forms) of cyclicity among those notions of cyclicity considered here.

\vskip3pt\noi
\begin{proposition}
Diagram\/ $1$ is complete.
\end{proposition}

\begin{proof}
(a) Classical examples$.$ Consider the (complex separable) Hilbert space
$\ell_+^2$ (of all complex-valued square-summable sequences)$.$ Let $S$ be the
(canonical) unilateral shift (of multiplicity one) on $\ell_+^2$ and take its
adjoint $S^*$; a backward unilateral shift on $\ell_+^2.$ Both $S$ and $S^*$
are cyclic operators, and while $S^{*2}$ is cyclic, $S^2$ is not cyclic
\cite[Problem 160]{Hal}, but $S$ is not supercyclic \cite[p.564]{HW}
(actually, every isometry is not supercyclic ---
\cite[Proof of Theorem 2.1]{AB}), while $S^*$ is supercyclic (in fact, the
adjoint of every injective unilateral weighted shift is supercyclic)
\cite[Theorem 3]{HW} but they are not hypercyclic (since $S$ and so $S^*\!$,
$S^2\!$, and $S^{*2}\!$ are \hbox{contractions}); \hbox{although} $2S^*$ is
hypercyclic \cite[Solution 168]{Hal} (and so $S^*$ must indeed be
super\-cyclic)$.$ This shows that there is no upward arrow on the left-hand
column.

\vskip4pt\noi
(b) There are weakly supercyclic operators that are not supercyclic
\cite[Corollary to Theorem 2.2]{San1}, and there are weakly hypercyclic
operator that are not hypercyclic \cite[Corollary 3.3]{CS}$.$ (Examples were
all built in \cite{San1} and \cite{CS} in terms of bilateral weighted shifts
on $\ell^p$ for ${2\le p<\infty}.$) This shows that there is no leftward
arrow between the two columns (except the lower row).

\vskip4pt\noi
(c$_1$) There are weakly supercyclic operators that are not weakly
hypercyclic$.$ \hbox{Indeed}, it was shown in \cite[Theorem 4.5]{San1} that
hyponormal operators (on Hilbert space) are not weakly hypercyclic (neither
supercyclic \cite[Theorem 3.1]{Bou}) but they can be weakly supercyclic$:$
there are weakly supercyclic unitary operators \cite[Example 3.6]{BM}.

\vskip4pt\noi
(c$_2$) Finally, to exhibit weakly cyclic operators that are not weakly
supercyclic pro\-ceed, for instance, as follows$.$ Every weakly supercyclic
hyponormal operator is a multiple of a unitary \cite[Theorem 3.4]{BM}$.$ Since
the canonical unilateral shift $S$, which is hyponormal, is cyclic (or,
equivalently, weakly cyclic), and since it is a completely nonunitary isometry
(and therefore not a multiple of a unitary), it is not weakly supercyclic$.$
(Another proof is exhibited in the forthcoming Proposition 4.1(b).)

\vskip4pt\noi
(c) Form (c$_1$) and (c$_2$), there is no upward arrow on the right-hand
column.
\end{proof}

\vskip3pt\noi
\begin{remark}
This refers to items in the proof of Proposition 3.1$.$ All examples in item
(a) were based on the unilateral shift$.$ However, a basic example of a
cyclic operator that is not supercyclic is given by normal operators$:$ no
normal operator is supercyclic \cite[p.564]{HW} (actually, no hyponormal
operator is supercyclic \cite[Theorem 3.1]{Bou}) and there exist cyclic normal
operators on separable Hilbert spaces (by the Spectral Theorem --- see, e.g.,
\cite[Proof of Theorem 3.11]{ST})$.$ Although all examples in item (b) were
built in \cite{CS} and \cite{San1} in terms of bilateral weighted shifts on
$\ell^p$ for ${2\le p<\infty}$, it was shown in \cite[Theorem 6.3]{MS} that a
bilateral weighted shifts on $\ell^p$ for ${1\le p<2}$ is weakly supercyclic
if and only if it is supercyclic$.$ The arguments used in items (a) and (c)
where based on hyponormal and cohyponormal operators (an operator ${T\in\BX}$
on a Hilbert space $\X$ is hyponormal if ${\|T^*x\|\le\|Tx\|}$ for every
${x\in\X}$, where ${T^*\!\in\BX}$ is the adjoint of $T\kern-1pt$, and
cohyponormal if its adjoint is hyponormal)$.$ Such a circle of ideas has been
extended from hyponormal to paranormal operators and beyond
\cite[Corollary 3.1]{Dug}, \cite[Theorem 2.7]{DKK} but we refrain from going
further than hyponormal operators here to keep up with the focus on weak
supercyclicity (and plain supercyclicity) only.
\end{remark}

\section{Auxiliary Results}

Items (a) and (b) of next lemma first appeared embedded in a proof of another
result in \cite[Proof of Theorem 2.1]{AB}$.$ The proof's argument is to show
that if an isometry has a supercyclic vector, then every vector is
supercyclic, which leads to a contradiction if $\X$ is a Banach space$.$ We
isolate this result in Lemma 4.1(a,b).

\vskip3pt\noi
\begin{lemma}
Let\/ $\X$ be an arbitrary\/ $($nonzero$)$ normed space.
\vskip4pt
\begin{description}
\item{$\kern-6pt$\rm(a)}
A supercyclic isometry on\/ $\X$ has no nontrivial invariant subspace.
\vskip4pt
\item{$\kern-6pt$\rm(b)}
An isometry on a complex Banach space is never supercyclic.
\vskip4pt
\item{$\kern-6pt$\rm(c)}
There exist isometries on a complex Hilbert space that are weakly supercyclic.
\end{description}
\end{lemma}

\begin{proof}
(a) Let ${V\!\in\BX}$ be an isometry on a normed space $\X$, which means
$\|V^nz\|=\|z\|$ for every ${z\in\X}$ and every integer ${n\ge0}.$ Suppose
$V\kern-1pt$ is supercyclic$.$ Let ${0\ne y\in\X}$ be a supercyclic vector for
$V\kern-1pt$ (with no loss of generality set $\|y\|=1$) and take an arbitrary
nonzero ${z\in\X}.$ Then there is a scalar-valued sequence of nonzero numbers
$\{\alpha_k\}_{k\ge0}$ such that ${\alpha_kV^{n_k}y\conv z}$ for some
subsequence $\{V^{n_k}\}_{k\ge0}$ of $\{V^n\}_{n\ge0}.$ Take an arbitrary
${\veps>0}$ so that
$$
\|\alpha_kV^{n_k}y-z\|<\veps
$$
for $k$ large enough$.$ Observe that $\{\alpha_k\}_{k\ge0}$ is bounded
(reason: since $V\kern-1pt$ is an isometry, $|\alpha_k|=\|\alpha_kV^{n_k}y\|$
and so boundedness of the convergent sequence $\{\alpha_kV^{n_k}y\}_{k\ge0}$
implies boundedness of $\{\alpha_k\}_{k\ge0}$)$.$ Thus set
${\alpha=\sup_n|\alpha_n|\kern-1pt>\kern-1pt0}.$ Take an arbitrary
${\delta>0}.$ Since the above displayed convergence holds for every
${0\ne z\in\X}$, take an arbitrary nonzero ${x\in\X}$ so that for every
${\delta>0}$ there exists a nonzero number $\beta$ and a positive integer
$m$ for which
$$
\|\beta\,V^my-x\|<\delta.
$$
Note that $\|x\|-\delta<|\beta|$ (in fact,
$\|x\|-|\beta|=\|x\|-\|\beta\,V^my\|\le \|{\beta\,V^my-x}\|<\delta$ since
$V\kern-1pt$ is an isometry)$.$ Moreover, by the above inequality, for
every ${n_k\ge m}$
$$
\|\beta\,V^{n_k}y-V^{n_k-m}x\|
=\|V^{n_k-m}(\beta\,V^my-x)\|
=\|\beta\,V^my-x\|<\delta.
$$
In particular, take any $\delta$ such that
$0<\delta<\frac{\veps\|x\|}{\alpha+\veps}.$ Thus
$\delta\alpha<\veps({\|x\|-\delta})<\veps|\beta|.$ Mul\-tiply both sides of
the above inequality by $\frac{|\alpha_k|}{|\beta|}$ to get
$$
\big\|\alpha_kV^{n_k}y-\smallfrac{\alpha_k}{\beta}V^{n_k-m}x\big\|
<\delta\smallfrac{|\alpha_k|}{|\beta|}\le\delta\smallfrac{\alpha}{|\beta|}
<\veps
$$
for every ${n_k\ge m}.$ Therefore, since $\|\alpha_kV^{n_k}y-z\|<\veps$ for
$k$ large enough,
$$
\big\|\smallfrac{\alpha_k}{\beta}V^{n_k-m}x-z\big\|
\le\big\|\smallfrac{\alpha_k}{\beta}V^{n_k-m}x-\alpha_kV^{n_k}y\big\|
+\|\alpha_kV^{n_k}y-z\|<2\kern1pt\veps
$$
for $k$ large enough, which means ${\frac{\alpha_k}{\beta}V^{n_k-m}x\conv z}$,
and so there exists a sequence $\{\alpha_j\}_{j\ge0}$ of nonzero numbers such
that ${\alpha_jV^{n_j}x\to z}$ for some subsequence $\{V^{n_j}\}_{j\ge0}$ of
$\{V^n\}_{n\ge0}.$ Since $z$ and $x$ are arbitrary nonzero vectors in $\X$,
this ensures that every vector in $\X$ is supercyclic for $V\kern-1pt$, and
hence $V\kern-1pt$ has no nontrivial invariant subspace
(cf$.$ Remark 3.1(g)).$\!$

\vskip6pt\noi
(b) The result in item (a) leads to a contradiction if $\X$ is a complex
Banach space because in this case isometries do have nontrivial invariant
subspaces$.$ In fact, if an isometry $V\kern-1pt$ on a Banach space is not
surjective, then $\R(V)$ is a nontrivial invariant (hyperinvariant, actually)
subspace for $V\kern-1pt$ because on a Banach space isometries have closed
range (see, e.g., \cite[Problem 4.41(d)]{EOT})$.$ On the other hand, since
isometries are always injective, if $V\kern-1pt$ is a surjective isometry
then it is invertible (whose inverse also is an isometry) so that
$\|V^n\|=\|V^{-n}\|=1$ for every ${n\ge0}.$ Thus surjective isometries are
power bounded with a power bounded inverse$.$ But {\it a nonscalar invertible
power bounded operator on a complex Banach space with a power bounded inverse
has a nontrivial invariant\/ $($hyperinvariant, {\rm $\kern-.5pt$actually)}
subspace}\/$.$ (See, e.g., \cite[Theorem 10.79]{AA} --- this is the Banach
space counterpart of a well-known result due to Sz.-Nagy which says$:$
{\it an invertible power bounded operator on a Hilbert space with a power
bounded inverse is similar to a unitary operator}\/; see, e.g.,
\cite[Corollary 1.16]{MDOT})$.$ Thus an isometry on a complex Banach
space has a nontrivial invariant subspace (see also \cite[Theorem J]{God})
and so it cannot be supercyclic according \hbox{to (a)}.

\vskip6pt\noi
(c) There are weakly supercyclic unitary operators on a complex Hilbert
space \cite[Example 3.6]{BM} (see also \cite[Theorem 2]{San2} and
\cite[Theorem 1.2]{Shk})$.$ Thus there are (invertible) weakly supercyclic
isometries on a Hilbert space.
\end{proof}

\vskip4pt
If $\X$ is a Hilbert space, then a completely nonunitary contraction is a
contraction such that no restriction of it to a reducing subspace is unitary
(i.e., such that every direct summand of if it is not unitary), and a
completely nonunitary isometry (i.e., a pure isometry) is precisely a
unilateral shift of some multiplicity$.$ These are consequences of
Nagy--Foia\c s--Langer decomposition for contractions and
von Neumann--Wold decomposition for isometries (see e.g., \cite[pp.3,8]{NF}
or \cite[pp.76,81]{MDOT}).

\vskip3pt\noi
\begin{proposition}
$\!${\rm(a)}
A weakly supercyclic isometry on a Hilbert space is unitary.
\vskip2pt\noi
{\rm(b)} Every unilateral shift on a Hilbert space is not weakly supercyclic.
\end{proposition}

\begin{proof}
(a) If a hyponormal operator is weakly supercyclic, then it is a multiple of
a unitary \cite[Theorem 3.4]{BM}$.$ Since isometries on a Hilbert space are
hyponormal with norm 1, then a weakly supercyclic isometry on a
Hilbert space is unitary.

\vskip6pt\noi
(b) A unilateral shift, of any multiplicity, on a Hilbert space is a
completely nonunitary isometry, thus not weakly supercyclic by item (a).
\end{proof}

\section{Weak and Strong Supercyclicity}

The weak counterpart of the supercyclicity criterion described in
Remark 3.1(c) was considered in \cite{BCS} (also in \cite{BM} implicitly),
and it was referred to as weak 1-se\-quential supercyclicity in \cite{Shk}$.$
Although there are reasons for such a terminology, we will change it here to
weak l-sequential supercyclicity, replacing the numeral ``1'' with the letter
``l'' for ``limit''$.$ A nonzero vector $y$ in $\X$ is a {\it weakly
l-sequentially supercyclic vector}\/ for an operator $T$ in $\BX$ if for
every ${x\in\X}$ there exists an $\FF$-valued sequence $\{\alpha_k\}_{k\ge0}$
(which depends on $x$ and $y$ and consists of nonzero numbers) such that for
some subsequence $\{T^{n_k}\}_{k\ge0}$ of $\{T^n\}_{n\ge0}$ the $\X$-valued
sequence $\{\alpha_kT^{n_k}y\}_{k\ge0}$ converges weakly to $x.$ That is,
$$
\alpha_kT^{n_k}y\wconv x.
$$
This means the projective orbit $\Oe_T(\span\{y\})$ of the vector $y$
under $T\kern-1pt$ is weakly l-sequentially dense in $\X$ in the following
sense$.$ For any set ${A\sse\X}$ let $A^{-wl}$ denote the set of all weak
limits of weakly convergent $A$-valued sequences, and $A$ is said to be
weakly l-sequentially dense in $\X$ if ${A^{-wl}=\X}.$ Thus $y$ is a weakly
l-sequentially supercyclic vector for $T$ if and only if
$$
\Oe_T(\span\{y\})^{-wl}=\X.
$$
An operator $T$ in $\BX$ is a {\it weakly l-sequentially supercyclic
operator}\/ if it has a weakly l-sequentially supercyclic vector$.$ Observe
that
\vskip9pt\noi
\centerline{
Supercyclic $\;\;\limply\;\;$
Weakly l-Sequentially Supercyclic $\;\;\limply\;\;$
Weakly Supercyclic,
}
\vskip9pt\noi
and the converses fail
(see, e.g., \cite[pp.38,39]{Shk}, \cite[pp.259,260]{BM2}).

\vskip6pt
We will be dealing with normed spaces $\X$ with the following property$:$ an
$\X$-val\-ued sequence $\{x_k\}_{k\ge0}$ converges strongly (i.e., in the
norm topology) if and only if it converges weakly and the norm sequence
$\{\|x_k\|\}_{k\ge0}$ converges to the limit's norm; that is,
${x_k\conv x}$ $\iff$ $\big\{{x_k\wconv x}$ and ${\|x_k\|\to\|x\|}\big\}.$
We say that a normed space $\X$ that has the above property is a normed space
of type $1.$ (Also called a Radon--Riesz space, whose property is called
the Radon--Riesz property --- see, e.g., \cite[Definition 2.5.26]{Meg})$.$
Hilbert spaces are Banach spaces of type $1$ $\!$\cite[Problem 20]{Hal}.

\vskip3pt\noi
\begin{theorem}
Suppose\/ $T$ is an operator on a type\/ $1$ normed space\/ $\X.$ If
\begin{description}
\item{$\kern-7pt$\rm(a)$\kern3pt$}
$T$ is power bounded,
\vskip2pt
\item{$\kern-7pt$\rm(b)$\kern3pt$}
${y\in\X}$ is a weakly l-sequentially supercyclic vector for\/ $T\kern-1pt$,
\vskip2pt
\item{$\kern-6pt$\rm(c)$\kern2pt$}
${y\in\X}$ is not a supercyclic vector for\/ $T\kern-1pt$,
\end{description}
then
\vskip0pt\noi
\begin{description}
\item{$\kern-7pt$\rm(d)$\kern3pt$}
every nonzero\/ ${f\in\X^*}$ is such that either
\vskip2pt
\begin{description}
\item{$\kern-6pt$\rm(d$_1$)$\kern3pt$}
$\liminf_n|f(T^ny)|=0,\;\;$ or
\vskip2pt
\item{$\kern-6pt$\rm(d$_2$)$\kern3pt$}
$\limsup_k|f(T^{n_k}y)|<\|f\|\,\limsup_k\|T^{n_k}y\|$ for some subsequence\/
\vskip1pt\noi
${\kern6pt}\{T^{n_k}\}_{k\ge0}$ of\/ $\{T^n\}_{n\ge0}$.
\end{description}
\end{description}
\end{theorem}

\begin{proof}
First we need the following auxiliary result.

\vskip6pt\noi
{\it Claim $1$}\/$.$
If ${z_k,z\in\X}$, where $\X$ is a normed space of type $1$, and if the
sequence $\{z_k\}_{k\ge0}$ is such that ${z_k\wconv z}$ and
${z_k\notconv z}$, then ${\|z\|<\limsup_k\|z_k\|}$.

\vskip6pt\noi
{\it Proof}\/$.$
If $\X$ is an arbitrary normed space, then ${z_k\wconv z}$ implies
$\|z\|\le{\liminf}_k\|z_k\|$ (see, e.g., \cite[Proposition 46.1]{Heu})$.$
Thus, if ${z_k\notconv z}$ and $\X$ is a normed space of type $1$, then
${\|z_k\|\not\to\|z\|}$ so that
$\|z\|\le{\liminf}_k\|z_k\|<\limsup_k\|z_k\|.\!\!\!\qed$

\vskip6pt\noi
Now consider assumptions (a), (b), and (c), and suppose (d) fails; that is,
suppose the contradictory ${\rm(d)\kern-9pt\big\slash}\;$ of (d) holds$:$
\begin{description}
\item{$\kern-7pt{\rm(d)\kern-9pt\big\slash}\kern5pt$}
there exists a nonzero ${f_0\in\X^*}$ such that
\vskip2pt
\begin{description}
\item{$\kern-8pt{\rm(d_1)\kern-14pt\big\slash}\kern12pt$}
$0<\liminf_n|f_0(T^ny)|\;\;$ and
\vskip2pt
\item{$\kern-8pt{\rm(d_2)\kern-14pt\big\slash}\kern12pt$}
$\limsup_k|f_0(T^{n_k}y)|=\|f_0\|\,\limsup_k\kern-1pt\|T^{n_k}y\|$
for every subsequence\/
\vskip1pt\noi
${\kern4pt}\{T^{n_k}\}_{k\ge0}$ of\/ $\{T^n\}_{n\ge0}$.
\end{description}
\end{description}
\vskip4pt\noi
According to assumption (b) let ${0\ne y\in\X}$ be a weakly l-sequentially
supercyclic vector for $T\kern-1pt.$ Thus for every ${x\in\X}$ there exists
a scalar-valued sequence $\{\beta_\ell\}_{\ell\ge0}$ (depending on $x$ and
$y$) such that
$$
\beta_\ell T^{n_\ell}y\wconv x
$$
for some subsequence $\{T^{n_\ell}\}_{\ell\ge0}$ of\/ $\{T^n\}_{n\ge0}.$ By
assumption (c) suppose $y$ is not a supercyclic vector for $T\kern-1pt.$
So there exists a nonzero vector ${x_0\in\X}$ such that
$$
\gamma_\ell T^{n_\ell}y\notconv x_0
$$
for every sequence of numbers $\{\gamma_\ell\}_{\ell\ge0}$ and every
subsequence $\{T^{n_\ell}\}_{\ell\ge0}$ of\/ $\{T^n\}_{n\ge0}.$ Then there is
a scalar-valued sequence $\{\alpha_j\}_{j\ge0}$ (depending on $x_0$ and $y$)
such that
$$
\alpha_jT^{n_j}y\wconv x_0
$$
for some subsequence $\{T^{n_j}\}_{j\ge0}$ of $\{T^n\}_{n\ge0}$, and
$$
\alpha_iT^{n_i}y\notconv x_0
$$
for every subsequence $\{T^{n_i}\}_{i\ge0}\!=\!\{T^{n_{j_i}}\}_{i\ge0}$ of
$\{T^{n_j}\}_{j\ge0}$ and every subsequence
$\{\alpha_i\}_{i\ge0}\!=\!\{\alpha_{j_i}\}_{i\ge0}$ of $\{\alpha_j\}_{j\ge0}.$
Next consider assumption ${{\rm(d)\kern-9pt\big\slash\kern3pt}}$ which
says$:$ there is a nonzero ${f_0\in\X^*}\!$ (which depends on $y$) satisfying
${{\rm(d_1)\kern-14pt\big\slash\kern9pt}}$ and
${{\rm(d_2)\kern-14pt\big\slash\kern9pt}}.$ Since
${\alpha_jT^{n_j}y\wconv x_0}$ we get
$|f(x_0)|=\lim_j|f(\alpha_jT^{n_j}y)|$ for every ${f\in\X^*}.$ In particular,
$$
|f_0(x_0)|={\lim}_j|f_0(\alpha_jT^{n_j}y)|={\lim}_j|\alpha_j|\,|f_0(T^{n_j}y)|.
$$
Hence ${\limsup_j|\alpha_j|<\infty}$ by
${{\rm(d_1)\kern-14pt\big\slash\kern9pt}}.$ (Indeed,
${0<\liminf_n|f_0(T^ny)|\in\RR}$ by (a) and
${{\rm(d_1)\kern-14pt\big\slash\kern9pt}}$ and so for every
${\veps\in(0,\liminf_n|f_0(T^ny)|})$ there exists a positive integer
$n_\veps$ such that if ${n\ge n_\veps}$ then ${\veps<|f_0(T^ny)|}$, and hence
${\limsup_j|\alpha_j|<\infty}$ since ${|f_0(x_0)|\in\RR}).$ Thus there is a
subsequence $\{\alpha_k\}_{k\ge0}\!=\!\{\alpha_{j_k}\}_{k\ge0}$ of
$\{\alpha_j\}_{j\ge0}$ such that
$$
\hbox{$\{|\alpha_k|\}_{k\ge0}$ converges}.
$$
Take this subsequence $\{\alpha_k\}_{k\ge0}$ of $\{\alpha_j\}_{j\ge0}$
and take the corresponding subsequence $\{T^{n_k}\}_{k\ge0}$ of
$\{T^{n_j}\}_{j\ge0}$ so that ${\alpha_kT^{n_k}y\wconv x_0}$ (because
${\alpha_jT^{n_j}y\wconv x_0}).$ Then
$$
\alpha_kT^{n_k}y\wconv x_0
\quad\;\hbox{and}\;\quad
\alpha_kT^{n_k}y\notconv x_0.
$$
Therefore, according to Claim 1,
$$
\|x_0\|<{\limsup}_k\|\alpha_kT^{n_k}y\|.
$$
 Again, since ${\alpha_kT^{n_k}y\wconv x_0}$, it follows that
$$
|f_0(x_0)|={\lim}_k|f_0(\alpha_kT^{n_k}y)|.
$$
Note$:$ if $\{\xi_k\}_{k\ge0}$ and $\{\zeta_k\}_{k\ge0}$ are bounded
sequences of nonnegative real numbers such that $\{\xi_k\}_{k\ge0}$
converges, then $\,{\lim_k\xi_k\limsup_k\zeta_k}={\limsup\xi_k\zeta_k}.$
(In fact, if ${\xi_k\!\to\xi}$ then for every ${\veps\!>\!0}$ there exists
an integer ${k_\veps\!\ge\!1}$ such that if ${k\!\ge\!k_\veps}$ then
${\xi\zeta_k-\xi_k\zeta_k}<\veps\zeta_k\le\veps\sup_k\zeta_k$, and hence
${\lim_k\xi_k\limsup_k\zeta_k}\le{\limsup\xi_k\zeta_k}
\le{\limsup_k\xi_k\limsup_k\zeta_k}={\lim_k\xi_k\limsup_k\zeta_k}).$
Thus, since ${\sup_k|f_0(T^{n_k}y)|<\infty}$ by assumption (a) and since
$\{|\alpha_k|\}_{k\ge0}$ converges, we get
\vskip2pt\noi
\begin{eqnarray*}
{\limsup}_k|\alpha_k|\,{\limsup}_k|f_0(T^{n_k}y)|
&\kern-6pt=\kern-6pt&
{\lim}_k|\alpha_k|\,{\limsup}_k|f_0(T^{n_k}y)|                         \\
&\kern-6pt=\kern-6pt&
{\limsup}_k|\alpha_k|\,|f_0(T^{n_k}y)|                                 \\
&\kern-6pt=\kern-6pt&
{\limsup}_k|f_0(\alpha_kT^{n_k}y)|                                     \\
&\kern-6pt=\kern-6pt&
{\lim}_k|\alpha_kf_0(T^{n_k}y)|.
\end{eqnarray*}
\vskip2pt\noi
Then, by the above three displayed expressions and
${\rm(d_2)\kern-14pt\big\slash\kern9pt}$,
\vskip2pt\noi
\begin{eqnarray*}
{\lim}_k|f_0(\alpha_kT^{n_k}y)|
&\kern-6pt=\kern-6pt&
|f_0(x_0)|\le\|f_0\|\,\|x_0\|                                          \\
&\kern-6pt<\kern-6pt&
\|f_0\|\,{\limsup}_k\|\alpha_kT^{n_k}y\|                               \\
&\kern-6pt\le\kern-6pt&
{\limsup}_k|\alpha_k|\,\|f_0\|\,{\limsup}_k\|T^{n_k}y\|                \\
&\kern-6pt=\kern-6pt&
{\limsup}_k|\alpha_k|\,{\limsup}_k|f_0(T^{n_k}y)|                      \\
&\kern-6pt=\kern-6pt&
{\lim}_k|f_0(\alpha_kT^{n_k}y)|,
\end{eqnarray*}
\vskip2pt\noi
which is a contradiction$.$ Therefore, (a), (b) and (c) imply (d).
\end{proof}

\vskip0pt\noi
\begin{corollary}
If a power bounded operator\/ $T$ on a type\/ $1$ normed space\/ $\X$ is
such that\/ $T$ is not supercyclic, then either
\begin{description}
\item{$\kern-4pt$\rm(i)$\kern2pt$}
$T$ is not weakly l-sequentially supercyclic,$\;\;$ or
\vskip2pt
\item{$\kern-6pt$\rm(ii)$\kern2pt$}
if\/ ${y\in\X}$ is a weakly l-sequentially supercyclic vector for\/
$T\kern-1pt$, then every nonzero\/ ${f\in\X^*}$ is such that either
\vskip2pt
\begin{description}
\item{$\kern9pt$}
$\liminf_n|f(T^ny)|=0,\;\;$ or
\vskip2pt
\item{$\kern9pt$}
$\limsup_k|f(T^{n_k}y)|<\|f\|\kern1pt\limsup_k\kern-1pt\|T^{n_k}y\|$ for some
subsequence\/
\vskip1pt\noi
$\{T^{n_k}\}_{k\ge0}$ of\/ $\{T^n\}_{n\ge0}$.
\end{description}
\end{description}
\end{corollary}

\begin{proof}
Immediate by Theorem 5.1.
\end{proof}

\vskip0pt\noi
\begin{remark}
This remark deals with operators on Hilbert spaces as it will be considered
in the forthcoming Corollaries 5.2 and 6.1$.$ Whenever we refer to a
\hbox{Hilbert} space, the inner product in it will be denoted by
${\<\;\,;\;\>}$.

\vskip4pt\noi
(a)
Although hyponormal operators are never supercyclic \cite[Theorem 3.1]{Bou},
neither weakly hypercyclic \cite[Theorem 4.5]{San1}, there exist weakly
l-sequentially supercyclic hyponormal operators$.$ In fact, every weakly
supercyclic (in particular, every weakly l-sequentially supercyclic)
hyponormal operator is a multiple of a unitary \cite[Theorem 3.4]{BM}, and
there exist weakly supercyclic (in fact, weakly l-sequentially supercyclic)
unitary operators \cite[Example 3.6, pp.10,12]{BM} (see also
\cite[Question 1]{Shk})$.$ Thus a weakly l-sequentially supercyclic hyponormal
contraction must be unitary$.$ Corollary 5.2 below gives a condition for a
hyponormal contraction (or a unitary operator) to be weakly l-sequentially
supercyclic.

\vskip4pt\noi
(b)
The above discussion guarantees the existence of weakly l-sequentially
supercyclic power bounded operators that are not supercyclic (see also
\cite[Corollary to Theorem 2.2]{San1})$.$ Thus alternative (ii) in
Corollary 5.1 cannot be dismissed$.$ Actually, the existence of weakly
l-sequentially supercyclic unitary operators (which are never supercyclic)
shows alternative (ii) in Corollaries 5.2 and 5.3 below cannot be dismissed
as well; that is, it shows the existence of unitary operators satisfying
condition (ii) in Corollaries 5.1, 5.2, and 5.3.
\end{remark}
\goodbreak\noi

\vskip3pt\noi
\begin{corollary}
If a power bounded operator\/ $T$ on a Hilbert space\/ $\X$ is hyponormal,
then it is a contraction and either
\begin{description}
\item{$\kern-4pt$\rm(i)$\kern2pt$}
$T$ is not weakly l-sequentially supercyclic,$\;\;$ or
\vskip2pt
\item{$\kern-6pt$\rm(ii)$\kern2pt$}
if\/ ${y\in\X}$ is a weakly l-sequentially supercyclic vector for\/
$T\kern-1pt$, then\/ $T\kern-1pt$ is unitary and every nonzero\/ ${z\in\X}$
is such that either
\vskip2pt
\begin{description}
\item{$\kern-4pt$}
$\liminf_n|\<T^ny\,;z\>|=0,\;\;$ or
\vskip2pt
\item{$\kern-4pt$}
$\limsup_k|\<T^{n_k}y\,;z\>|<\|z\|\kern1pt\|y\|$ for some subsequence\/
$\{T^{n_k}\}_{k\ge0}$ of\/ $\{T^n\}_{n\ge0}$.
\end{description}
\end{description}
\end{corollary}

\begin{proof}
It is well known that if $T$ is hyponormal, then it is normaloid (i.e.,
$\|T\|^n=\|T^n\|$ for every ${n\kern-1pt\ge\kern-1pt1}$), and every power
bounded normaloid operator is a contraction$.$ A hyponormal operator on a
Hilbert space is not supercyclic \cite[Theorem 3.1]{Bou}$.$ Then apply
Corollary 5.1 (replacing $f(x)$ with $\<{x\,;z}\>$ according to the Riesz
Representation Theorem in Hilbert space), and recall that a weakly supercyclic
hyponormal contraction is unitary (cf$.$ Remark 5.1), thus an isometry.
\end{proof}

\vskip3pt\noi
\begin{corollary}
If\/ $T$ is an isometry on a type\/ $1$ Banach space\/ $\X$, then either
\begin{description}
\item{$\kern-4pt$\rm(i)$\kern2pt$}
$T$ is not weakly l-sequentially supercyclic,$\;\;$ or
\vskip2pt
\item{$\kern-6pt$\rm(ii)$\kern2pt$}
if\/ ${y\in\X}$ is a weakly l-sequentially supercyclic vector for\/
$T\kern-1pt$, then every nonzero\/ ${f\in\X^*}$ is such that either
\vskip2pt
\begin{description}
\item{$\kern-4pt$}
$\liminf_n|f(T^ny)|=0,\;\;$ or
\vskip2pt
\item{$\kern-4pt$}
$\limsup_k|f(T^{n_k}y)|<\|f\|\kern1pt\|y\|$ for some subsequence\/
$\{T^{n_k}\}_{k\ge0}$ of\/ $\{T^n\}_{n\ge0}$.
\end{description}
\end{description}
\end{corollary}

\begin{proof}
If $T$ is an isometry on a Banach space, then it is not supercyclic
\cite[Proof of Theorem 2.1]{AB} (Lemma 4.1)$.$ Thus apply Corollary 5.1$.$
(In a Hilbert space setting this is a particular case of Corollary 5.2, where
$T$ is an invertible isometry).
\end{proof}

\section{Weak Supercyclicity and Stability}

It was proved in \cite[Theorem 2.1]{AB} that {\it a power bounded operator\/
$T$ on a Banach space\/ $\X$ such that\/ $\|T^nx\|\not\to0$ for every}\/
${0\ne x\in\X}$ (i.e., a power bounded operator of class $\C_{1\cdot}$)
{\it has no supercyclic vector}\/$.$ The next result is a weak version of it.

\vskip3pt\noi
\begin{theorem}
If a power bounded operator\/ $T$ on a type\/ $1$ normed space\/ $\X$ is
such that\/ ${T^nx\notwconv0}$ for every\/ ${0\ne x\in\X}$, then either
\begin{description}
\item{$\kern-4pt$\rm(i)$\kern2pt$}
$T$ has no weakly l-sequentially supercyclic vector,$\;\;$ or
\vskip2pt
\item{$\kern-6pt$\rm(ii)$\kern2pt$}
if\/ ${y\in\X}$ is a weakly l-sequentially supercyclic vector for\/
$T\kern-1pt$, then every nonzero\/ ${f\in\X^*}$ for which\/
${f(T^ny)\not\to0}$ is such that either
\vskip2pt
\begin{description}
\item{$\kern9pt$}
$\liminf_n|f(T^ny)|=0,\;\;$ or
\vskip2pt
\item{$\kern9pt$}
$\limsup_k|f(T^{n_k}y)|<\|f\|\kern1pt\limsup_k\|T^{n_k}y\|$ for some
subsequence\/
\vskip1pt\noi
$\{T^{n_k}\}_{k\ge0}$ of\/ $\{T^n\}_{n\ge0}$.
\end{description}
\end{description}
\end{theorem}

\begin{proof}
Consider the following result.

\vskip6pt\noi
{\it Claim $1$}\/$.$
If a power bounded operator on any normed space is such that
${T^nx\notwconv0}$ for every\/ ${0\ne x\in\X}$, then it has no supercyclic
vector.

\vskip6pt\noi
{\it Proof}\/$.$
If an operator $T$ on a normed space $\X$ is such that ${T^nx\notwconv0}$ for
some (for every) ${0\ne x\in\X}$, then it is clear that ${T^nx\notconv0}$ for
some (for every) ${0\ne x\in\X}$ (strong convergence implies weak convergence
to the same limit)$.$ It was proved in \cite[Theorem 2.1]{AB} that a if power
bounded operator $T$ on a \hbox{Banach} space $\X$ is such that
${\|T^nx\|\not\to0}$ for every ${0\ne x\in\X}$, then it has no supercyclic
vector, whose proof survives in any normed space.$\!\!\!\qed$

\vskip6pt\noi
Thus under the theorem hypothesis, Claim 1 ensures $T$ has no supercyclic
vector$.$ If, in addition, $\X$ is a type $1$ normed space and $T$ does not
satisfy condition (i) --- that is, if $T$ has a weakly l-sequentially
supercyclic vector $y$ --- then condition (ii) holds by Theorem 5.1 (or
Corollary 5.1).
\end{proof}

\vskip2pt
It was proved in \cite[Theorem 2.2]{AB} by using \cite[Theorem 2.1]{AB} that
{\it a Banach-space supercyclic power bounded operator is strongly stable}\/,
whose proof in fact does not require completeness$.$ Theorem 6.2 below is a
weak version of it based on Theorem 6.1$.$ Weakly l-sequentially supercyclic
contractions on Hilbert space are character\-ized in Corollary 6.1 as a
consequence of Theorem 6.2.

\vskip3pt\noi
\begin{theorem}
If a power bounded operator\/ $T$ on a type\/ $1$ normed space\/ $\X$ is
weakly l-sequentially supercyclic, then either
\begin{description}
\item{$\kern-4pt$\rm(i)$\kern2pt$}
$T$ is weakly stable,$\;\;$ or
\vskip2pt
\item{$\kern-6pt$\rm(ii)$\kern2pt$}
if\/ ${y\in\kern-1pt\X}$ is a weakly l-sequentially supercyclic vector for\/
$T\kern-1pt$ such that\/ ${T^ny\!\notwconv0}$, then for every nonzero\/
${f\in\X^*}$ such that\/ ${f(T^ny)\not\to0}$ either
\vskip2pt
\begin{description}
\item{$\kern9pt$}
$\liminf_n|f(T^ny)|=0,\;\;$ or
\vskip2pt
\item{$\kern9pt$}
$\limsup_k|f(T^{n_k}y)|<\|f\|\kern-1pt\limsup_k\|T^{n_k}y\|$ for some
subsequence\/
\vskip1pt\noi
$\{T^{n_k}\}_{k\ge0}$ of\/ $\{T^n\}_{n\ge0}$.
\end{description}
\end{description}
\end{theorem}

\begin{proof}
First we show: if (i) fails, then there is a weakly l-sequentially
supercyclic vector $y$ such that ${T^ny\notwconv0}.$ That is, if
${T^nx\notwconv0}$ for some ${x\in\X}$, then the set
$$
\big\{y\in\kern-1pt\X\!:\,\hbox{$y$ is a weakly l-sequentially supercyclic
vector for\/ $T\kern-1pt$ such that\/ ${T^ny\!\notwconv0}$}\/\big\}
$$
is nonempty.

\vskip6pt\noi
{\it Claim $1$}\/$.$
Suppose $T$ is a power bounded weakly l-sequentially supercyclic operator on
a normed space $\X.$ If there exists a vector ${x\in\X}$ such that
${T^nx\notwconv0}$, then there exists a weakly l-sequentially supercyclic
vector ${y\in\X}$ for $T$ such that ${T^ny\notwconv0}$.

\vskip6pt\noi
{\it Proof}\/$.$
Let ${Y\kern-1pt\sse\X}$ denote the set of all weakly l-sequentially
supercyclic vectors for an operator $T$, and so $T$ is weakly l-sequentially
supercyclic if and only if ${Y\kern-1pt\ne\void}.$ The same argument of
Remark 3.1(a) ensures ${\Oe_T(\span\{y\})\\\0\sse Y}\kern-1pt$ for every
${y\in Y}$ (see \cite[Lemma 5.1]{Kub2})$.$ Since
${(\Oe_T(\span\{y\})\\\0)^{-wl}}\!=\X$ for every ${y\in Y}$ (definition of weak
l-sequential supercyclicity), then ${Y\ne\void}\limply\!{Y^{-wl}\!=\X}.$
However, more is true$.$ Denseness is attained in the norm topology (cf$.$
\cite[Theorem 5.1]{Kub2}):
$$
Y\ne\void
\quad\limply\quad
Y^-\!=\X.
$$
(A weak version of the above implication was considered in \cite[Proposition
2.1]{San1}, where $Y$ is replaced by the set of all weakly supercyclic vectors
--- see Remark 3.1(e))$.$ Take an arbitrary $x$ in $\X.$ If $Y^-\!=\X$, then
there exists a $Y\!$-valued sequence $\{y_k\}$ such that ${\|y_k\!-x\|\to0}.$
If ${T_ny\wconv0}$ for every ${y\in Y}\!$, which means ${f(T_ny)\to0}$ for
every $f$ in the dual $\X^*\!$ of $\X$ and every $y$ in $Y\!$, then
${|f(T_ny_k)|\to0}$ for every $f$ in $\X^*\!$ and every integer $k.$
Therefore since
\goodbreak\noi
$$
|f(T_nx)|\le|f(T_n(y_k-x))|+|f(T_ny_m)|
\le\|f\|\,{\sup}_n\|T_n\|\kern1pt\|y_k-x\|+|f(T_ny_k)|
$$
for every ${f\in\X^*\!}$ and every ${x\in\X}$, we get ${T_nx\wconv0}$ for
every ${x\in\X}.$ So if there is an ${x\in\X}$ such that ${T^nx\notwconv0}$,
then there is a ${y\kern-1pt\in\kern-1ptY}$ such that
${T^ny\notwconv0}.\!\!\!\qed$

\vskip6pt\noi
Now let $T$ be an operator on a type $1$ normed space $\X$ and consider the
following assumptions.
\begin{description}
\item{$\kern-7pt$(a)$\kern3pt$}
$T$ is power bounded, and set $\beta=\sup_n\|T^n\|>0$.
\vskip2pt
\item{$\kern-7pt$(b)$\kern3pt$}
$T$ is weakly l-sequentially supercyclic.
\vskip2pt
\item{$\kern-7pt$(c)$\kern3pt$}
If\/ ${y\in\X}$ is a weakly l-sequentially supercyclic vector for\/ $T$ such
that ${T^ny\notwconv0}$, then for some ${f_0\in\X^*}$ with $\|f_0\|=1$ such
that\/ ${f_0(T^ny)\not\to0}$,
\vskip2pt
\begin{description}
\item{$\kern-8pt$(c$_1$)$\kern5pt$}
$0<\liminf_n|f_0(T^ny)|\;\;$ and
\vskip2pt
\item{$\kern-8pt$(c$_2$)$\kern5pt$}
$\limsup_k|f_0(T^{n_k}y)|=\|f_0\|\,\limsup_k\|T^{n_k}y\|$ for every
subsequence\/
\vskip1pt\noi
${\kern4pt}\{T^{n_k}\}_{k\ge0}$ of\/ $\{T^n\}_{k\ge0}$.
\end{description}
\end{description}

\vskip2pt\noi
{\it Claim $2$}\/$.$
If ${y\in\X}$ is weakly l-sequentially supercyclic for $T\kern-1pt$,
then ${T^ny\wconv0}$.

\vskip6pt\noi
{\it Proof}\/$.$
Under assumption (b) there exists a weakly l-sequentially supercyclic unit
vector ${y\in\X}$ (${\|y\|=1}$) for $T\kern-1pt.$ Suppose
$$
T^ny\notwconv0.
$$
Under assumptions (a) and (c) Theorem 6.1 says there exists a unit vector
$v$ (i.e., ${\|v\|=1}$) in $\X$ such that ${T^nv\wconv0}.$ Since $y$ is
weakly l-sequentially supercyclic, there is a sequence $\{\alpha_k\}_{k\ge0}$
of nonzero numbers such that ${\alpha_kT^{n_k}y\wconv v}$ for some subsequence
$\{T^{n_k}\}_{k\ge0}.$ So, ${f(\alpha_kT^{n_k}y)\to f(v)}$ for every
${f\in\X^*\!}.$ Take a unit vector ${f\in\X^*}$ (${\|f\|=1}$) for which
${\frac{1}{2}\le|f(v)|\le1}$ (recall$:$ $1=\|v\|=\sup_{\|f\|=1}|f(v)|).$
Thus there exists a positive integer $k_f$ such that if ${k\ge k_f}$ then
$\big|\,|f(\alpha_kT^{n_k}y)|-|f(v)|\,\big|
\le|f(\alpha_kT^{n_k}y)-f(v)|<\frac{|f(v)|}{2}$,
and hence $|f(v)|-|f(\alpha_kT^{n_k}y)|<\frac{|f(v)|}{2}$, which implies
$\frac{1}{4}\le\frac{|f(v)|}{2}=|f(v)|-\frac{|f(v)|}{2}<|f(\alpha_kT^{n_k}y)|
\le\|f\||\alpha_k|\sup_n\|T^n\|\|y\|=\beta|\alpha_k|$
accord\-ing to (a)$.$ Thus, for $k$ large enough,
$$
\smallfrac{1}{4\beta}<|\alpha_k|.
$$
Now take any unit vector ${f_0\in\X^*}$ (${\|f_0\|=1}$) satisfying assumption
(c) so that, according to (c$_1$), there exists a positive number $\delta$
such that
$$
\delta<{\liminf}_n|f_0(T^ny)|.
$$
Next take any positive $\gamma$ such that ${\gamma<\frac{\delta}{4\beta^2}}.$
Take an arbitrary integer ${m\ge0}.$ Note that
$|f_0({\alpha_kT^{n_k+m}y-T^mv})|=|f_0(T^m({\alpha_kT^{n_k}y-v}))|=
|(T^{m*}f_0)({\alpha_kT^{n_k}y-v})|=|f_m({\alpha_kT^{n_k}y-Tv})|$
for $f_m={T^{m*}f_0\in\X^*}$, where ${T^{m*}\in\B[\X^*]}$ is the normed-space
adjoint of ${T^m\in\BX}$ (see, e.g., \cite[Section 3.2]{Sch})$.$ Since
${\alpha_kT^{n_k}y\wconv v}$, there is a positive integer $k_m$ such that if
${k\ge k_m}$ then $|f_m({\alpha_kT^{n_k}y-v})|<\beta\smallfrac{\gamma}{2}.$
Thus, for any ${m\ge0}$ and $k$ large enough,
$$
|f_0({\alpha_kT^{n_k+m}y-T^mv})|<\beta\smallfrac{\gamma}{2}.
$$
Finally, since ${T^nv\wconv0}$, take $m$ sufficiently large such that
$$
|f_0(T^mv)|<\beta\smallfrac{\gamma}{2}.
$$
Then, by the above four displayed inequalities, for $k$ and $m$
large enough,
\goodbreak\noi
\begin{eqnarray*}
\smallfrac{\delta}{4\beta}\!
&\kern-6pt<\kern-6pt&
\!\smallfrac{{\liminf}_k|f_0(T^{n_k+m}y)|}{4\beta}                     \\
&\kern-6pt<\kern-6pt&
{\liminf}_k|f_0(T^{n_k+m}y)|{\inf}_k|\alpha_k|
\le{\liminf}_k|f_0(\alpha_kT^{n_k+m}y)|
\phantom{\lim_k}                                                       \\
&\kern-6pt\le\kern-6pt&
{\liminf}_k|f_0(\alpha_kT^{n_k+m}y-T^mv)|+{\liminf}_k|f_0(T^mv)|
\phantom{\lim_*}                                                       \\
&\kern-6pt<\kern-6pt&
\beta\smallfrac{\gamma}{2}+\beta\smallfrac{\gamma}{2}
=\beta\gamma<\smallfrac{\delta}{4\beta},
\end{eqnarray*}
which is a contradiction$.$ Therefore if ${y\in\X}$ is weakly l-sequentially
supercyclic for $T\kern-1pt$, then ${T^ny\wconv0}.\!\!\!\qed$

\vskip6pt\noi
By Claim 1 (which depends on assumptions (a) and (b)) if ${T^ny\wconv0}$ for
every weakly l-sequentially supercyclic vector $y$ in $\X$ for $T\kern-1pt$,
then ${T^nx\wconv0}$ for every $x$ in $\X.$ So the result in Claim 2 (which
depends on assumptions (a), (b) and (c)) ensures $T$ is weakly stable$.$
Thus, under assumptions (a) and (b), $T$ is weakly stable if assumption (c)
holds; that is if assumptions (a) and (b) hold and if $T$ is not weakly
stable, then assumption (c) fails; equivalently, assumption (ii) holds.
\end{proof}

\vskip2pt
For a power bounded operator supercyclicity implies strong stability
\cite[Theorem 2.2]{AB}$.$ Theorem 6.2 prompts the question$.$ Consider a
power bounded operator $T.$ {\it Does weak l-sequential
supercyclicity implies weak stability}\/?
\vskip4pt\noi
$$
\newmatrix{
T{\rm\;\;is\;\:supercyclic} &\limply\;\; & T^n\sconv O     \phantom{\Big|} \cr
\big\Downarrow              &            & \big\Downarrow   \cr
T{\rm\;\;is\;\:weakly}\;\;l{\rm-sequentially\;\;supercyclic}
                    & \;\;\;\;\query\;\; & \;\;T^n\wconv O \phantom{\Big|}.\cr}
$$
\vskip4pt\noi
In particular, {\it can alternative}\/ (ii) {\it be dismissed from
Theorems}\/ 6.1 {\it an}\/d 6.2 ?

\vskip3pt\noi
\begin{corollary}
If a contraction\/ $T$ on a Hilbert space is weakly l-sequentially
supercyclic, then either
\begin{description}
\item{$\kern-4pt$\rm(i)$\kern2pt$}
$T$ is weakly stable,$\;\;$ or
\vskip2pt
\item{$\kern-6pt$\rm(ii)$\kern2pt$}
$\!$if\/ $y$ is a weakly l-sequentially supercyclic vector for the unitary
part\/ $U\kern-1pt$ of\/ $T$ such that\/ ${U^ny\notwconv0}$, then for every
nonzero\/ $z$ such that\/ ${\<U^ny\,;z\>\not\to0}$ either
\vskip2pt
\begin{description}
\item{$\kern-2pt$}
$\liminf_n|\<U^ny\,;z\>|=0,\;\;$ or
\vskip2pt
\item{$\kern-2pt$}
$\limsup_k|\<U^{n_k}y\,;z\>|<\|z\|\kern1pt\|y\|$ for some subsequence\/
$\{U^{n_k}\}_{k\ge0}$ of\/ $\{U^n\}_{n\ge0}$.
\end{description}
\end{description}
\end{corollary}

\begin{proof}
Let $T$ be a contraction on a Hilbert space $\X.$ By the
Nagy--Folia\c s--Langer decom\-position for Hilbert-space contractions (see,
e.g., \cite[p.8]{NF} or \cite[p.76]{MDOT}), $\X$ admits an orthogonal
decomposition $\X={\U^\perp\!\oplus\U}$, where $T$ is uniquely a direct sum
of a completely nonunitary contraction $C={T|_{\U^\perp}\in\B[\U]}$ and a
unitary operator $U={T|_\U\in\B[\U]}$ (where any of these parcels may be
missing):
$$
T=C\oplus U,
$$
where $C$ is the completely nonunitary part of $T$ and $U\kern-1pt$ is the
unitary part of $T\kern-1pt.$ Every completely nonunitary contraction is
weakly stable (see, e.g., \cite[p.55]{Fil} or \cite[p.106]{MDOT})$.$ Thus $T$
is weakly stable if and only if $U\kern-1pt$ is weakly stable; that is,
$$
C^n\wconv O
\qquad\hbox{and}\qquad
T^n\wconv O
\;\;\,\hbox{if and only if}\,\;\;
U^n\wconv O.
$$
Suppose $U\kern-1pt$ acts on a nonzero space (otherwise the result is
trivially verified)$.$ If $T={C\oplus U}$ is weakly l-sequentially supercyclic
(or supercyclic), then both $C$ and $U\kern-1pt$ are weakly l-sequentially
supercyclic$.$ Thus the result follows by Theorem 6.2 and by the Riesz
Representation Theorem in Hilbert space, since $U\kern-1pt$ is an isometry.
\end{proof}

\vskip2pt
Weak l-sequential supercyclicity and weak stability for unitary operators
are discussed in \cite[Theorem 5.1]{Kub1} in terms of a condition similar to
the so-called angle criterion for supercyclicity --- see, e.g.,
\cite[Theorem 9.1]{BM2}.

\vskip-13pt\noi
\bibliographystyle{amsplain}

\end{document}